\documentclass[11pt]{article}
\usepackage{latexsym,amsfonts,amssymb,amsmath,amsthm}
\usepackage{graphicx}

\usepackage[usenames,dvipsnames]{color}
\usepackage{ulem}
\usepackage{xcolor}

\newcommand {\rd}{\color{red}}

\begin{document}
\setlength{\baselineskip}{16pt}

\parindent 0.5cm
\evensidemargin 0cm \oddsidemargin 0cm \topmargin 0cm \textheight 22.5cm \textwidth 16cm \footskip 2cm \headsep
0cm

\newtheorem{theorem}{Theorem}[section]
\newtheorem{lemma}{Lemma}[section]
\newtheorem{proposition}{Proposition}[section]
\newtheorem{definition}{Definition}[section]
\newtheorem{example}{Example}[section]
\newtheorem{corollary}{Corollary}[section]

\newtheorem{remark}{Remark}[section]

\numberwithin{equation}{section}

\def\p{\partial}
\def\I{\textit}
\def\R{\mathbb R}
\def\C{\mathbb C}
\def\u{\underline}
\def\l{\lambda}
\def\a{\alpha}
\def\O{\Omega}
\def\e{\epsilon}
\def\ls{\lambda^*}
\def\D{\displaystyle}
\def\wyx{ \frac{w(y,t)}{w(x,t)}}
\def\imp{\Rightarrow}
\def\tE{\tilde E}
\def\tX{\tilde X}
\def\tH{\tilde H}
\def\tu{\tilde u}
\def\d{\mathcal D}
\def\aa{\mathcal A}
\def\DH{\mathcal D(\tH)}
\def\bE{\bar E}
\def\bH{\bar H}
\def\M{\mathcal M}

\def\disp{\displaystyle}
\def\undertex#1{$\underline{\hbox{#1}}$}
\def\card{\mathop{\hbox{card}}}
\def\sgn{\mathop{\hbox{sgn}}}
\def\exp{\mathop{\hbox{exp}}}
\def\OFP{(\Omega,{\cal F},\PP)}
\newcommand\JM{Mierczy\'nski}

\newcommand\Q{\ensuremath{\mathbb{Q}}}
\newcommand\Z{\ensuremath{\mathbb{Z}}}
\newcommand\N{\ensuremath{\mathbb{N}}}

\title{Stabilization in two-species chemotaxis   systems
with singular sensitivity and Lotka-Volterra competitive kinetics}
\author{Halil Ibrahim Kurt and Wenxian Shen   \\
Department of Mathematics and Statistics\\
Auburn University\\
Auburn University, AL 36849\\
U.S.A. }

\date{}
\maketitle

\begin{abstract}
The current paper is concerned with the stabilization in
 the following parabolic-parabolic-elliptic chemotaxis system with singular sensitivity and Lotka-Volterra competitive kinetics,
 \begin{equation}
\label{abstract-eq}
\begin{cases}
u_t=\Delta u-\chi_1 \nabla\cdot (\frac{u}{w} \nabla w)+u(a_1-b_1u-c_1v) ,\quad &x\in \Omega\cr
v_t=\Delta v-\chi_2 \nabla\cdot (\frac{v}{w} \nabla w)+v(a_2-b_2v-c_2u),\quad &x\in \Omega\cr
0=\Delta w-\mu w +\nu u+ \lambda v,\quad &x\in \Omega \cr
\frac{\p u}{\p n}=\frac{\p v}{\p n}=\frac{\p w}{\p n}=0,\quad &x\in\p\Omega,
\end{cases}
\end{equation}
where $\Omega \subset \mathbb{R}^N$ is a bounded smooth  domain, and   $\chi_i$, $a_i$, $b_i$, $ c_i$ ($i=1,2$)  and
$\mu,\, \nu, \, \lambda$
 are positive constants.  In \cite{HKWS3}, among others,
 we proved that  for any given nonnegative initial data $u_0,v_0\in C^0(\bar\Omega)$ with $u_0+v_0\not \equiv 0$, \eqref{abstract-eq} has
a unique globally defined classical solution   $(u(t,x;u_0,v_0),v(t,x;u_0,v_0),w(t,x;u_0,v_0))$ with
$u(0,x;u_0,v_0)=u_0(x)$ and $v(0,x;u_0,v_0)=v_0(x)$  {provided that $\min\{a_1,a_2\}$ is large relative to $\chi_1,\chi_2$, and $u_0+v_0$ is not  small in the case that
$(\chi_1-\chi_2)^2\le \max\{4\chi_1,4\chi_2\}$ and $u_0+v_0$ is neither { small} nor big in the case that $(\chi_1-\chi_2)^2>\max\{4\chi_1,4\chi_2\}$.}
In this paper, we first  assume that the competition in \eqref{abstract-eq} is weak in the sense that
$$
\frac{c_1}{b_2}<\frac{a_1}{a_2},\quad \frac{c_2}{b_1}<\frac{a_2}{a_1}.
$$
Then \eqref{abstract-eq} has a unique positive constant solution
$(u^*,v^*,w^*)$, where
$$
u^*=\frac{a_1b_2-c_1a_2}{b_1b_2-c_1c_2},\quad v^*=\frac{b_1a_2-a_1c_2}{b_1b_2-c_1c_2}, \quad w^*=\frac{\nu}{\mu}u^*+\frac{\lambda}{\mu} v^*.
$$
We obtain some  explicit conditions on $\chi_1,\chi_2$  which ensure that the positive constant solution $(u^*,v^*,w^*)$ is globally stable { in the sense that} for any given nonnegative initial data $u_0,v_0\in C^0(\bar\Omega)$ with $u_0\not \equiv 0$ and $v_0\not \equiv 0$, { if $(u(t,x;u_0,v_0),v(t,x;u_0,v_0),w(t,x;u_0,v_0))$ exists globally and stays bounded, then}
$$
\lim_{t\to\infty}\Big(\|u(t,\cdot;u_0,v_0)-u^*\|_\infty
+\|v(t,\cdot;u_0,v_0)-v^*\|_\infty+\|w(t,\cdot;u_0,v_0)-w^*\|_\infty\Big)=0.
$$
\end{abstract}


\medskip


\section{Introduction}
\label{S:intro}

Chemotaxis models are used to  describe  the movements of cells or organisms in response to chemicals in their environments in many biological processes such as population dynamics, tumor growth, gravitational collapse, embryo development, etc. Since the pioneering works by Keller and Segel (\cite{ke-se1}, \cite{Keller-1}) on chemotaxis models, many studies have been conducted on the qualitative properties of various chemotaxis models such as the  blow-up in finite time, global existence and boundedness of classical solutions, existence of steady states, and the asymptotic behavior of globally defined   solutions, etc.  The reader is referred to \cite{bel-wi,  hillen-painter,Hor} and the references therein for some detailed introduction
into the mathematics of chemotaxis models.

 The current paper is devoted to the study of stabilization in   the following parabolic-parabolic-elliptic chemotaxis system with singular sensitivity and   Lotka-Volterra competitive kinetics,
\begin{equation}
\label{main-eq}
\begin{cases}
u_t=\Delta u-\chi_1 \nabla\cdot (\frac{u}{w} \nabla w)+u(a_1-b_1u-c_1v) ,\quad &x\in \Omega { ,\, \,\,\quad t>0,} \cr
v_t=\Delta v-\chi_2 \nabla\cdot (\frac{v}{w} \nabla w)+v(a_2-b_2v-c_2u),\quad &x\in \Omega { ,\, \,\,\quad t>0,}\cr
0=\Delta w-\mu w +\nu u+ \lambda v,\quad &x\in \Omega {, \,\,\,\quad t>0,} \cr
\frac{\p u}{\p n}=\frac{\p v}{\p n}=\frac{\p w}{\p n}=0,\quad &x\in\p\Omega { , \quad t>0,}
\end{cases}
\end{equation}
where $\Omega \subset \mathbb{R}^N$ is a bounded smooth  domain, and   $\chi_i$, $a_i$, $b_i$, $ c_i$ ($i=1,2$)  and
$\mu,\, \nu, \, \lambda$
 are positive constants. 
Biologically,  \eqref{main-eq} describes   the evolution of  two competitive  species  subject to the influence of a chemical substance produced by themselves. In such case, 
$u(t,x)$ and $v(t,x)$ represent the population densities
of two competitive   species, whose local dynamics is governed by  the Lotka-Volterra competitive kinetics
$u(a_1-b_1u-c_1v)$ and
$v(a_2-b_2v-c_2u)$.    $w(t,x)$  represents  the concentration of
the chemical substance, which is produced by the two species $u$ and $v$  at the rates $\nu$ and $\lambda$, respectively, and degrades at the rate $\mu$.  $\frac{\chi_1}{w}$ and $\frac{\chi_2}{w}$  reflect the strength of the chemical substance on the movements of  the two species, and are referred to as chemotaxis sensitivity functions or coefficients.
 Note  that  the chemotaxis sensitivities $\frac{\chi_i}{w}$  ($i=1,2$) are  singular near $w=0$,   reflecting  an inhibition of chemotactic migration at high signal concentrations. Such a sensitivity  describing  the living organisms' response to the chemical signal  was derived by the Weber-Fechner law (see \cite{Keller-1}).

There are many studies on various two competing species chemotaxis systems. For example,
consider {the following  two competing species chemotaxis system,}
\begin{equation}
\label{main-eq0}
\begin{cases}
u_t= \Delta u-  \nabla\cdot (u \chi_1(w) \nabla w)+u(a_1-b_1u-c_1v) ,\quad &x\in \Omega\cr
v_t= \Delta v- \nabla\cdot (v\chi_2(w)\nabla w)+ v(a_2 -b_2v- c_2u),\quad &x\in \Omega\cr
\tau w_t= \Delta w-{\mu w}  +\nu u+ \lambda  v,\quad &x\in \Omega\cr
\frac{\p u}{\p n}=\frac{\p v}{\p n}=\frac{\p w}{\p n}=0,\quad &x\in\p\Omega,
\end{cases}\,
\end{equation}
where, as in \eqref{main-eq},  $\Omega\subset\R^N$ is a bounded smooth domain;
$u(t,x)$ and $v(t,x)$ represent the population densities
of two competitive   species;   $w(t,x)$  represents  the concentration of
the chemical substance produced by the two species; $\tau\ge 0$ is related to the diffusion rate of the chemical substance;  and   the functions  $\chi_1(w)$ and $\chi_2(w)$ reflect the strength of the chemical substance on the movements of the two species, and are referred to as chemotaxis sensitivity functions or coefficients.

 When $\chi_1(w)\equiv \chi_1 >0$ and $\chi_2(w)\equiv \chi_2>0 $, \eqref{main-eq0} is referred to as  a two species chemotaxis system with {\it  linear sensitivity}.
Consider  \eqref{main-eq0} with linear sensitivity
 and $\tau=0$. It is known that  if $N\le 2$, or $N\ge 3$  and $\chi_i$ is relatively small with respect  to $b_i$ and $c_i$ ($i=1,2$), then  \eqref{main-eq0} possesses a unique globally defined classical solution $(u(t,x;u_0,v_0),v(t,x;u_0,v_0), w(t,x;u_0,v_0))$ with
$u(0,x;u_0,v_0)=u_0(x)$ and $v(0,x;u_0,v_0)=v_0(x)$  for  any nonnegative initial data $u_0,v_0\in C^0(\bar\Omega)$  (see  \cite{isra,  issh3, limuwa, st-te-wi, te-wi} and references therein).
In \cite{te-wi},  the {\it stabilization} at the unique positive constant solutions in \eqref{main-eq0} is proved when the competition is weak (roughly speaking,  $c_1$ and $c_2$ are  small) and $\chi_1,\chi_2$ are relatively small,
that is,  all the positive solutions converge to the unique positive constant solution
(see \cite[Theorem 0.1]{te-wi} for the detail).  The authors of \cite{st-te-wi} proved that {\it competitive exclusion}  occurs when $c_1$ is small and $c_2$ is large (resp.  $c_1$ is large and $c_2$ is small) and $\chi_1,\chi_2$ are relatively small,  that is,
{ $v(t,x;u_0,v_0)\to 0$}  (resp. { $u(t,x;u_0,v_0)\to 0$}) as $t\to\infty$ for any positive initials $u_0,v_0$  (see \cite[Theorem 1.1]{st-te-wi} for the detail).
We refer  the reader to the papers  \cite{bllami, isra, issh2, issh3, miz3, miz4}, etc. for
other studies on  the large time behaviors  of  globally defined classical solutions of \eqref{main-eq0}  with linearity sensitivity and $\tau=0$.

Consider \eqref{main-eq0} with linear sensitivity  and $\tau=1$.
It is also known  that, if
$N\le 2,$ or $N\ge 3$ and   $\chi_1$ and $\chi_2$ are relatively small with respect  to other parameters in \eqref{main-eq0}, then
  \eqref{main-eq0} possesses a unique globally defined classical solution $(u(t,x;u_0,v_0 ,w_0),v(t,x;u_0,v_0{ ,w_0}), w(t,x;u_0,v_0{ ,w_0}))$ with
$u(0,x;u_0,v_0{ ,w_0})=u_0(x)$, $v(0,x;u_0,v_0{ ,w_0})=v_0(x)$, and $w(0,x;u_0,v_0{,w_0})=w_0(x)$  for any nonnegative initial data $u_0,v_0\in C(\bar\Omega)$, $w_0\in W^{1,\infty}(\Omega)$, (see \cite{ba-wi, limuwa, zhni}, etc.).
The large-time behaviors of globally defined classical solutions are also  investigated  in \cite{ba-wi, zhni}, etc.
 The reader is referred  to the articles \cite{Hor2, zhni} for the further details.

Among nonlinear sensitivity functions   are $\chi_i(w)=\frac{\chi_i}{w}$ for some positive constant $\chi_i$ ($i=1,2)$. 
Note that \eqref{main-eq0} with  $\chi_i(w)=\frac{\chi_i}{w}$
and $\tau=0$
reduces to \eqref{main-eq}. 
The current paper is to study the long time dynamics of   \eqref{main-eq}.

For the biological reason, we are only interested in  the  solutions of \eqref{main-eq} with initial functions
$u_0,v_0\in C^0(\bar\Omega)$ with $u_0\ge 0$, $v_0\ge 0$, and  $\int_\Omega(u_0+v_0)>0$.  We remark that
 if $v_0=0$ (resp. $u_0=0$), then $v(t,x)\equiv 0$
(resp. $u(t,x)\equiv 0$) on the existence interval. If $v(t,x)\equiv 0$, then \eqref{main-eq} becomes
\begin{equation}
\label{main-eq1}
\begin{cases}
u_t=\Delta u-\chi_1 \nabla\cdot (\frac{u}{w} \nabla w)+u(a_1-b_1u) ,\quad &x\in \Omega\cr
0=\Delta w-\mu w +\nu u,\quad &x\in \Omega \cr
\frac{\p u}{\p n}=\frac{\p w}{\p n}=0,\quad &x\in\p\Omega,
\end{cases}
\end{equation}
and if $u(t,x)\equiv 0$, then \eqref{main-eq} becomes
\begin{equation}
\label{main-eq2}
\begin{cases}
v_t=\Delta v-\chi_2 \nabla\cdot (\frac{v}{w} \nabla w)+v(a_2-b_2v),\quad &x\in \Omega\cr
0=\Delta w-\mu w + \lambda v,\quad &x\in \Omega \cr
\frac{\p v}{\p n}=\frac{\p w}{\p n}=0,\quad &x\in\p\Omega.
\end{cases}\,
\end{equation}
Systems \eqref{main-eq1} and \eqref{main-eq2} are essentially the {same} and have been studied in  several works  (see \cite{Bil,Bla, CaWaYu, FuWiYo1,FuWiYo,HKWS,HKWS2,NaSe}, etc.). In particular, it is proved in \cite{HKWS} that for any $u_{0}\in C(\bar\Omega)$ with $u_{0}\ge 0$ and
$\int_\Omega u_{0}(x)dx>0$, \eqref{main-eq1} has a unique globally defined classical solution $(u(t,x;u_0),w(t,x;u_0))$ satisfying $u(0,x;u_0)=u_0(x)$
{ provided that $a_1$ is large relative to $\chi_1$ and $u_0$ is not small} (see \cite[Theorem 1.2]{HKWS}).
Cao et al. in \cite{CaWaYu} considered the stability of the positive constant solution of \eqref{main-eq1} with $\mu=\nu=1$ and proved that,    if
\begin{equation*}
a_1>2(\sqrt{\chi_1+1}-1)^2+\frac{\chi_1^2}{16\eta|\Omega|}
\end{equation*}
 with the constant $\eta$ depending  on $\Omega$,  then  for any nonnegative initial data $0\not\equiv u_0\in C^0(\bar \Omega)$ satisfying that
\begin{equation}
\label{assumption-eq3}
\int_{\Omega} u_0^{-1}<16 \eta b_1 |\Omega|^2/\chi_1^2,
\end{equation}
  the globally defined positive solution of \eqref{main-eq}  converges to $\big( \frac{a_1}{b_1},\frac{\nu}{\mu} \frac{a_1}{b_1} \big)$ as $t\to\infty$ exponentially (see \cite[Theorem 1]{CaWaYu} for details).
When $0<\chi\ll 1$, by the arguments \cite[Theorem 1]{CaWaYu}, it can  also be proved that the positive constant solution $(\frac{a_1}{b_1},\frac{\nu}{\mu}\frac{a_1}{b_1})$
of \eqref{main-eq1}  is globally stable. In the very recent paper \cite{HKWSSX}, the authors of the current paper together with S. Xue showed that there are nonconstant positive stationary solutions  of \eqref{main-eq1} bifurcating from  $(\frac{a_1}{b_1},\frac{\nu}{\mu}\frac{a_1}{b_1})$ when $\chi$ is not small.

 As far as we know,  there is not much work on the global existence and asymptotic behavior of solutions of \eqref{main-eq} with initial conditions  $ { ( } u_0,v_0)$ satisfying
\begin{equation}
\label{old-initial-cond-eq}
u_0,v_0\in C(\bar\Omega),\,\, u_0,v_0\ge 0,\,\, {\rm and}\,\,\int_\Omega u_0>0,\,\, \int_\Omega v_0>0.
\end{equation}
Recently, we  investigated  in \cite{HKWS3}  the global existence, boundedness, and combined persistence  of  classical solutions of
\eqref{main-eq} (see Definition \ref{classical-solution-def} for the definition of classical solutions of \eqref{main-eq}).
Among others, we proved that for any given $ {( } u_0,v_0)$ satisfying {\eqref{old-initial-cond-eq}},
\eqref{main-eq} has a unique globally {defined} classical solution
 $(u(t,x;u_0,v_0),v(t,x;u_0,v_0),w(t,x;u_0,v_0))$ satisfying
\begin{equation}
\label{local-2-eq0}
\lim_{t\to 0+}\|u(t,\cdot;u_0,v_0)-u_0(\cdot)\|_{C^0(\bar\Omega)}=0,\quad \lim_{t\to 0+}\|v(t,\cdot;u_0,v_0)-v_0(\cdot)\|_{C^0(\bar\Omega)}=0
\end{equation}
{provided that $a_{\min}=\min\{a_1,a_2\}$ is large relative to $\chi_1$ and $\chi_2$, and $u_0+v_0$ is not small in the case that $(\chi_1-\chi_2)^2\le \max\{4\chi_1,4\chi_2\}$
and 
$u_0+v_0$ is neither small nor big in  the case that
$(\chi_1-\chi_2)^2>\min\{4\chi_1,4\chi_2\}$
(see Theorem 1.3 in \cite{HKWS3}).  

The objective of this paper is to study  stabilization in \eqref{main-eq} when the competition is  weak in the sense that
\begin{equation}
\label{weak-competition-eq1}
{c_1}<\frac{a_1b_2}{a_2},\quad {c_2}<\frac{a_2 b_1}{a_1}.
\end{equation}

Assume \eqref{weak-competition-eq1}. Then \eqref{main-eq} has a unique positive constant solution
$(u^*,v^*,w^*)$, where
\begin{equation}
\label{positive-constant-solu-eq}
u^*=\frac{a_1b_2-c_1a_2}{b_1b_2-c_1c_2},\quad v^*=\frac{b_1a_2-a_1c_2}{b_1b_2-c_1c_2}, \quad w^*=\frac{\nu}{\mu}u^*+\frac{\lambda}{\mu} v^*.
\end{equation}
It is well known that $(u^*,v^*)$ is a globally  stable solution of the following competition diffusion system,
\begin{equation}
\label{main-no-diffusion-eq}
\begin{cases}
u_t=\Delta u+u(a_1-b_1u-c_1v) ,\quad &x\in \Omega\cr
v_t=\Delta v+v(a_2-b_2v-c_2u),\quad &x\in \Omega\cr
\frac{\p u}{\p n}=\frac{\p v}{\p n}=0,\quad &x\in\p\Omega,
\end{cases}
\end{equation}
that is, for any given positive initial functions $u_0,v_0$,
$$
\lim_{t\to\infty}\big(\|u(t,\cdot;u_0,v_0)-u^*\|_\infty+\|v(t,\cdot;u_0,v_0)-v^*\|_\infty\big)=0,
$$
where $(u(t,x,;u_0,v_0),v(t,x;u_0,v_0))$ is the solution of \eqref{main-no-diffusion-eq} with
$u(0,x;u_0,v_0)=u_0(x)$ and $v(0,x;u_0,v_0)=v_0(x)$.

Consider \eqref{main-eq}. It is interesting to identify the parameter regions of $\chi_1,\chi_2$ for the occurrence of stabilization in \eqref{main-eq}, i.e., for the global stability of $(u^*,v^*,w^*)$ { in certain sense}.
Among others, we will prove

\begin{itemize}

\item {\it Assume \eqref{weak-competition-eq1}. Then $(u^*,v^*,w^*)$ is a globally  stable stationary solution of \eqref{main-eq} {  in the sense that any globally defined bounded positive solution converges to $(u^*,v^*,w^*)$ as $t\to\infty$}  provided that $\chi_1,\chi_2$ are small relative to the other parameters in \eqref{main-eq}} (see Theorem \ref{main-thm2}).

\end{itemize}

The rest of this paper is organized as follows.  In section 2, we  introduce some standing notations to be used in the paper, state the main results of the paper, and provide some remarks on the main results. In section 3, we recall some of the known results from \cite{HKWS3}  to be used in this current work. In section 4, we  study the lower bounds of $w(t,x;u_0,v_0)$, which plays an important role in the study of stabilization in \eqref{main-eq}.  Section 5 is devoted to the analysis of stabilization of system \eqref{main-eq}.

\section{Notations, main results, and remarks}

In this section, we introduce some notations to be used in the paper, state the main results of the paper,  and provide some remarks on the main results.

\subsection{Notations}

In this subsection, we  introduce some  notations to be used in the rest of the paper.
Let
\begin{equation*}
\begin{cases}
    a_{\min}=\min\{a_1,a_2\},\quad a_{\max}=\max\{a_1,a_2\}\cr
b_{\min}=\min\{b_1,b_2\}, \quad  \, b_{\max}=\max\{b_1,b_2\}\cr
c_{\min}=\min\{c_1,c_2\},\quad \, c_{\max}=\max\{c_1,c_2\}.
\end{cases}
\end{equation*}
Let
\begin{equation}
\label{m-0-eq}
m_0=m_0(a_1,a_2,b_1,b_2,c_1,c_2)= \big(b_{\max}+c_{\max}\big)  \cdot \max\Big\{1,\frac{a_1}{b_1}+\frac{a_2}{b_2}\Big\},
\end{equation}
\begin{equation}
\label{m-star-eq}
m^*(\mu,\nu,\lambda)=\max\Big\{\max\big\{1,-\mu+\frac{\nu}{2}+\frac{\lambda}{2}\big\}\frac{2\nu ^2}{\mu^2} +\frac{\nu}{2},\max \big\{1,-\mu+\frac{\nu}{2}+\frac{\lambda}{2}\big\}\frac{2\lambda ^2}{\mu^2} +\frac{\lambda}{2}\Big\},
\end{equation}
and
\begin{equation*}
m(\mu,\nu,\lambda)=\frac{m^*{ (}\mu,\nu,\lambda)}{\min\{\nu,\lambda\}}.
\end{equation*}
 Let $|\Omega|$ be the Lebesgue measure of $\Omega$, and
\begin{equation}
\label{delta-0-eq}
\delta_0=\int_0^\infty \frac{1}{(4\pi t)^{n/2}}e^{-\Big(t+\frac{ ({\rm diam}\Omega)^2}{4t}\Big)}dt.
\end{equation}

For given $\xi_1,\xi_2,\eta>0$,  let
\begin{equation*}
B(\xi_1,\xi_2,\eta)=\left(\begin{matrix}\xi_1 b_1-\eta & \frac{1}{2}(\xi_1 c_1+\xi_2 c_2)
\\
& \\
\frac{1}{2}{ (}\xi_1 c_1+\xi_2 c_2) & \xi_2 b_2-\eta
\end{matrix}\right).
\end{equation*}
Let
\begin{equation}
\label{xi-1-2-eq1}
\begin{cases}
\xi_1=1,\quad \xi_2=\frac{c_1}{c_2}\quad {\rm if}\,\,  c_2\ge c_1\cr
\xi_1=\frac{c_2}{c_1},\quad  \xi_2=1\quad {\rm if}\,\,
c_1>c_2.
\end{cases}
\end{equation}
Assume \eqref{weak-competition-eq1}. We have  $c_1c_2<b_1b_2$ and the function
\begin{equation*}
g(\eta):=\eta^2 -(\xi_1b_1+\xi_2 b_2)\eta  +\xi_1\xi_2 (b_1b_2-c_1c_2)
\end{equation*}
is either positive for every $\eta>0$ or there is $\tilde \eta=\tilde\eta(a_1,a_2,b_1,b_2,c_1,c_2)>0$ such that
\begin{equation}
\label{tidle-eta-eq}
g(\tilde\eta)=0\quad {\rm and}\quad g(\eta)>0\quad \forall\, 0<\eta<\tilde\eta.
\end{equation}
Let
\begin{equation*}
\eta_0=\eta_0(a_1,a_2,b_1,b_1,c_1,c_2)=\begin{cases}
\min\{\xi_1 b_1,\xi_2 b_2\}\quad &{\rm if}\,\, g(\eta)>0\,\,\, \forall\, \eta>0\cr
\cr
\min\{\xi_1b_1,\xi_2 b_2,\tilde \eta\}\quad &{\rm if\,\, there\,\, is}\,\,  \tilde\eta>0\,\, {\rm satisfying}\,\,  \eqref{tidle-eta-eq}.
\end{cases}
\end{equation*}

It is not difficult to prove that the matrix $B(\xi_1,\xi_2,\eta)$  is positive definite for any
$0<\eta<\eta_0$. In fact, it is easy to see that
$$
{\rm Trace}\Big( B(\xi_1,\xi_2,\eta)\Big)=\xi_1b_1+\xi_2 b_2-2\eta>0\quad \forall \,\, 0<\eta<\eta_0.
$$
Note that
\begin{align*}
{\rm det}\Big(B(\xi_1,\xi_2,\eta)\Big)&=\big(\xi_1 b_1-\eta\big)\big(\xi_2 b_2-\eta\big)-\frac{1}{4}\big(\xi_1 c_1+\xi_2 c_2\big)^2\nonumber\\
&=\eta^2-(\xi_1b_1+\xi_2b_2)\eta+\xi_1\xi_2b_1b_2-\frac{1}{4}\Big(\xi_1^2 c_1^2 +2\xi_1\xi_2 c_1c_2+\xi_2^2 c_2^2)\nonumber\\
&=\eta^2-(\xi_1b_1+\xi_2b_2)\eta+\xi_1\xi_2(b_1b_2-c_1c_2)-\frac{1}{4}\Big(\xi_1^2 c_1^2 -2\xi_1\xi_2 c_1c_2+\xi_2^2 c_2^2)\nonumber\\
&=\eta^2-(\xi_1b_1+\xi_2b_2)\eta +\xi_1\xi_2 (b_1b_2-c_1c_2)-\frac{1}{4}(\xi_1 c_1-\xi_2 c_2)^2.
\end{align*}
By \eqref{xi-1-2-eq1}, we have $\xi_1 c_1-\xi_2 c_2=0$,  and then
$$
{\rm det}\big(B(\xi_1,\xi_2,\eta)\big)=g_1(\eta)>0\quad \forall\, \, 0<\eta<\eta_0\le \tilde \eta.
$$
Therefore,  $B(\xi_1,\xi_2,\eta)$  is positive definite for any
$0<\eta<\eta_0$.

\subsection{Main results}

In this subsection, we state the main results of this paper. {Note that this paper is to study the asymptotic behavior of globally defined bounded positive solutions of \eqref{main-eq}, that is, solutions of \eqref{main-eq} with initial function $(u_0,v_0)$ satisfying}
\begin{equation}
\label{new-initial-cond-eq}
{ (u_0,v_0)\,\,\, \text{satisfies}\,\,\, \eqref{old-initial-cond-eq},\,\,\, T_{\max}(u_0,v_0)=\infty,\,\,\, \limsup_{t\to\infty}\|(u+v)(t,\cdot;u_0,v_0)\|_\infty<\infty.}
\end{equation}
{The existence of such solutions of \eqref{main-eq} was studied in
\cite{HKWS3}.}

The first main result  is  on the lower bounds of $w(t,x;u_0,v_0)$ and is stated as follows, which  is  one of the main ingredients in  the study of stabilization  in \eqref{main-eq}.

\begin{theorem}
\label{main-thm1}
\begin{itemize}
\item[(1)]
Assume that $\chi_1,\chi_2$ satisfy
\begin{equation}
\label{cond-on-chi-eq1}
 a_{\min}>2 \mu\cdot h(\chi_1,\chi_2),
\end{equation}
where
\begin{equation}
\label{h-function-eq}
h(\chi_1,\chi_2)=  \max\Big\{\chi_1,\chi_2, \frac{(\chi_1-\chi_2)^2}{4}\Big\}.
\end{equation}
{Then,} for any given $u_0,v_0$ satisfying {\eqref{new-initial-cond-eq}},
\begin{equation}
\label{new-lower-bound-eq1}
\liminf_{t\to\infty}\int_\Omega (u+v)\ge  \frac{|\Omega|\Big (a_{\min}-{ 2}h(\chi_1,\chi_2)\Big)^{1/p}}{m_0^{1/p}},
\end{equation}
and
\begin{equation}
\label{lower-bound-proof-eq3}
\liminf_{t\to\infty}\inf_{x\in\Omega} w(t,x;u_0,v_0)\ge \delta_0  \cdot\min\{\nu,\lambda\}\cdot
 \frac{|\Omega|\Big (a_{\min}-2\mu \cdot h(\chi_1,\chi_2)\Big)^{1/p}}{m_0^{1/p}},
\end{equation}
where
\begin{equation}
\label{p-eq}
p=\begin{cases}
1\quad &{\rm if}\quad \frac{(\chi_1-\chi_2)^2}{4\chi_2}\le 1\quad {\rm or}\quad \frac{(\chi_1-\chi_2)^2}{4 \chi_1}\le 1\cr\cr
\frac{4\chi_2}{(\chi_1-\chi_2)^2}\quad &{\rm if}\quad  \frac{(\chi_1-\chi_2)^2}{4\chi_2}> 1\quad {\rm and}\quad \frac{(\chi_1-\chi_2)^2}{4 \chi_1}>1.
\end{cases}
\end{equation}

\item[(2)] If $\chi_1=\chi_2$,  assume that
\begin{equation}
\label{cond-on-chi-eq1-1}
a_{\min}>2\mu { \big(\sqrt{\chi_1+1}-1\big)^2}.
\end{equation}
Then,  for any given $u_0,v_0$ satisfying {\eqref{new-initial-cond-eq}},
\begin{equation}
\label{new-lower-bound-eq2}
\liminf_{t\to\infty} \int_\Omega (u+v)\ge |\Omega|\frac{a_{\min}-2\mu{ (\sqrt{\chi_1+1}-1)^2}}{b_{\max}+c_{\max}},
\end{equation}
 and
\begin{equation}
\label{lower-bound-proof-eq3-1}
\liminf_{t\to\infty}\inf_{x\in\Omega} w(t,x)\ge \delta_0\cdot\min\{\nu,\lambda\}\cdot   \frac{|\Omega| \Big(a_{\min}-2\mu{ (\sqrt{\chi_1+1}-1)^2}\Big)}{b_{\max}+c_{\max}}.
\end{equation}
\end{itemize}
\end{theorem}

The second main result is on the stabilization, i.e.,  the global stability of $(u^*,v^*,w^*)$, and is stated in the following theorem.

\begin{theorem}
\label{main-thm2}
Assume that $c_1,c_2$ satisfy \eqref{weak-competition-eq1}.
 Let $(u^*,v^*,w^*)$ be as in  \eqref{positive-constant-solu-eq}.

\begin{itemize}
\item[(1)] Assume that
  $\chi_1,\chi_2$ satisfy
\begin{equation}
\label{cond-on-chi-eq2}
a_{\min}>2\mu \cdot  h(\chi_1,\chi_2) +m_0\cdot  m_1 \cdot \Big(\frac{\chi_1^2 u^*+\chi_2^2 v^*}{4}\Big)^p,
\end{equation}
where  $h(\chi_1,\chi_2)$ is as in \eqref{h-function-eq}, $m_0$ is as in \eqref{m-0-eq},  $p$ is as in \eqref{p-eq}, and
\begin{equation*}
m_1=\max\Big\{ 1, \frac{m(\mu,\nu,\lambda)}{\delta_0 \cdot |\Omega|\cdot   w^*\cdot  \eta_0 }\Big\}.
\end{equation*}
Then, for any $u_0,v_0$ satisfying {\eqref{new-initial-cond-eq}},
$$
\lim_{t\to\infty} \Big(\|u(t,\cdot;u_0,v_0)-u^*\|_\infty+\|v(t,\cdot;u_0,v_0)-v^*\|_\infty\Big)=0.
$$

\item[(2)] If $\chi_1=\chi_2$,
 assume that
\begin{equation}
\label{cond-on-chi-eq2-1}
a_{\min}>2\mu { \big(\sqrt{\chi_1+1}-1\big)^2} +\tilde m_0\cdot \tilde  m_1 \cdot \Big(\frac{\chi_1^2 u^*+\chi_2^2 v^*}{4}\Big),
\end{equation}
where
\begin{equation*}
\tilde m_0=b_{\max}+c_{\max},\quad \text{and} \quad
\tilde m_1=\frac{m(\mu,\nu,\lambda)}{\delta_0 \cdot |\Omega|\cdot   w^*\cdot \eta_0}.
\end{equation*}
Then, for any $u_0,v_0$ satisfying {\eqref{new-initial-cond-eq}},
$$
\lim_{t\to\infty} \Big(\|u(t,\cdot;u_0,v_0)-u^*\|_\infty+\|v(t,\cdot;u_0,v_0)-v^*\|_\infty\Big)=0.
$$
\end{itemize}
\end{theorem}

\subsection{Remarks}

In this subsection, we provide some remarks on the main results stated in subsection 2.2.

\begin{remark}
\label{lower-bound-rk}
We make the following comments on Theorem \ref{main-thm1}.
\begin{itemize}

\item[(1)] { By the third equation in (1.1), it is not difficult to prove  that  \eqref{lower-bound-proof-eq3} and \eqref{lower-bound-proof-eq3-1} follow from \eqref{new-lower-bound-eq1} and \eqref{new-lower-bound-eq2}, respectively (see Lemma \ref{lower-bound-lm2}). 
 Hence Theorem \ref{main-thm1}  is mainly about the combined mass persistence, that is,  $\liminf_{t\to\infty}\int_\Omega (u+v)>0$
(see \eqref{new-lower-bound-eq1} and \eqref{new-lower-bound-eq2}). 
Note that   $a_1,a_2>0$ are the intrinsic growth rates of the species $u$ and $v,$ respectively. Biologically, the results in  Theorem \ref{main-thm1}  can be interpreted  as follows: 
when the chemotaxis sensitivities $\chi_1,\chi_2$ and the degradation rate of the chemical substance are small relative to the   intrinsic growth rates of the species 
(see \eqref{cond-on-chi-eq1} and \eqref{cond-on-chi-eq1-1}),   the populations of two species  as a whole   (i.e. the sum of the total populations of two species)    always persist.
It should be pointed out that the weak competition condition \eqref{weak-competition-eq1}
is not assumed in Theorem \ref{main-thm1}.  When the competition between two species is not weak, it is possible that one of the species goes extinct.}

\item[(2)] The lower bound of $w(t,x;u_0,v_0)$  obtained in Theorem \ref{main-thm1} plays
an important role in the proof of Theorem \ref{main-thm2}.

\item[(3)] As it is mentioned in (1), to obtain a lower bound for $w(t,x;u_0,v_0)$, it is sufficient to obtain a lower bound for $\int_\Omega (u(t,x;u_0,v_0)+v(t,x;u_0,v_0))dx$ (see Lemma \ref{lower-bound-lm2}). To obtain a lower bound for  $\int_\Omega (u(t,x;u_0,v_0)+v(t,x;u_0,v_0))dx$,
it suffices to obtain an upper bound for $\int_\Omega (u(t,x;u_0,v_0)+v(t,x;u_0,v_0))^{-p}$ for some $p>0$ (see Lemma \ref{lower-bound-lm3}).

\item[(4)]  In \cite{HKWS3},  it is proved { that $(u_0,v_0)$ satisfies \eqref{new-initial-cond-eq} provided that $a_{\min}=\min\{a_1,a_2\}$ is large relative to $\chi_1$ and $\chi_2$, and $u_0+v_0$ is not small in the case that $(\chi_1-\chi_2)^2\le \max\{4\chi_1,4\chi_2\}$ and $u_0+v_0$ is neither small nor big in the case that $(\chi_1-\chi_2)^2>\max\{4\chi_1,4\chi_2\}$ (see \cite[Theorem 1.3]{HKWS3}).}

\item[(5)]  When $\chi_1=\chi_2$, the condition \eqref{cond-on-chi-eq1-1} is weaker than
\eqref{cond-on-chi-eq1}, and the lower bound  in \eqref{lower-bound-proof-eq3-1} for $w(t,x;u_0,v_0)$  is larger than the lower bound in \eqref{lower-bound-proof-eq3} for
$w(t,x;u_0,v_0)$. Due to the singularity in \eqref{main-eq}, larger lower bound for $w(t,x;u_0,v_0)$ is better.

\end{itemize}
\end{remark}

\begin{remark}
\label{stability-rk}
We make the following comments on Theorem \ref{main-thm2}.
\begin{itemize}

\item[(1)] Note that the condition  \eqref{weak-competition-eq1} indicates that the competition between two species is weak.  
Under this condition,  \eqref{main-eq} has a unique positive constant solution
$(u^*,v^*,w^*)$ given by 
\eqref{positive-constant-solu-eq}. When $\chi_1=\chi_2=0$,  $(u^*,v^*, w^*)$ is globally stable, or  system \eqref{main-eq} is stabilized at $(u^*,v^*,w^*)$,  since $(u^*,v^*)$  is a globally stable solution of the competition diffusion system \eqref{main-no-diffusion-eq}. 
 Biologically, the results in  Theorem \ref{main-thm2}  can be interpreted  as follows: 
when the competition between two species is weak and
the chemotaxis sensitivities $\chi_1,\chi_2$ and the degradation rate of the chemical substance are small relative to the   intrinsic growth rates of the species 
(see \eqref{cond-on-chi-eq2} and \eqref{cond-on-chi-eq2-1}),  system \eqref{main-eq} is stabilized at 
the unique positive constant solution $(u^*,v^*,w^*)$.

\item[(2)] The condition \eqref{cond-on-chi-eq2} implies \eqref{cond-on-chi-eq1}
(resp.  the condition \eqref{cond-on-chi-eq2-1} implies \eqref{cond-on-chi-eq1-1}).

\item[(3)] The ideas to prove Theorem \ref{main-thm2} include  the following:
first prove
\begin{equation}
\label{l-2-convergence-eq1}
\lim_{t\to\infty} \Big(\|u(t,\cdot;u_0,v_0)-u^*\|_{L^2(\Omega)}+\|v(t,\cdot;u_0,v_0)-v^*\|_{L^2(\Omega)}\Big)=0,
\end{equation}
and then prove
$$
\lim_{t\to\infty} \Big(\|u(t,\cdot;u_0,v_0)-u^*\|_\infty+\|v(t,\cdot;u_0,v_0)-v^*\|_\infty\Big)=0.
$$
Note that \eqref{l-2-convergence-eq1} is equivalent to
\begin{equation}
\label{l-2-convergence-eq2}
\lim_{t\to\infty}\|u(t,\cdot;u_0,v_0)-u^*\|_{L^2(\Omega)}=0,\quad \lim_{t\to\infty} \|v(t,\cdot;u_0,v_0)-v^*\|_{L^2(\Omega)}=0.
\end{equation}
But we realize that it is not easy to prove \eqref{l-2-convergence-eq2} directly.

\item[(4)]  To prove \eqref{l-2-convergence-eq1}, we  first prove that the  energy function
$E(t)=E(t;u_0,v_0)$ is decreasing in $t$ for $t\gg 1$, where
\begin{align}
\label{energy-function-eq}
E(t;u_0,v_0)&=\xi_1\int_\Omega \Big(u(t,x;u_0,v_0)-u^*-u^*\ln\frac{u(t,x;u_0,v_0)}{u^*}\Big)\nonumber\\
&\quad +\xi_2\int_\Omega(v(t,x;u_0,v_0)-v^*-v^*\ln\frac{v(t,x;u_0,v_0)}{v^*}\Big),
\end{align}
where $\xi_1,\xi_2$ are as in \eqref{xi-1-2-eq1}.
Observe that energy functions of the form \eqref{energy-function-eq} are used in the study of  the stability of nonnegative stationary solutions of competitive systems without chemotaxis (see  \cite{Goh, NiShWa}, etc.).
Observe also that  $E(t;u_0,v_0)$ is well defined and $E(t;u_0,v_0)>0$ for all $t>0$.
 It should be pointed out that, in \cite{CaWaYu},
the following energy function
$$
 E_1(t)=  \int_{\Omega} \left[ u(t,x;u_0) -\frac{a_1}{b_1}-\frac{a_1}{b_1} \ln{\Big(\frac{b_1 u(t,x;u_0)}{a_1} \Big)} \right]dx
$$
is used to prove the stability of  the constant solution $(\frac{a_1}{b_1},\frac{\nu}{\mu}\frac{a_1}{b_1})$ of \eqref{main-eq1} in the case that $N=2$  with $u_0$ satisfying \eqref{assumption-eq3} (see \cite[Theorem 1]{CaWaYu} for details). In Theorem \ref{main-thm2}, except {\eqref{new-initial-cond-eq}}, no any other condition is assumed for $u_0,v_0$,
which is due to the lower bound we obtained in Theorem \ref{main-thm1} for any $u_0,v_0$ satisfying {\eqref{new-initial-cond-eq}}.
\end{itemize}
\end{remark}

\section{Preliminary}

In this section, we recall some results from \cite{HKWS3} to be used in this paper.

First of all, we give the definition of classical solutions of \eqref{main-eq}.

\begin{definition}
\label{classical-solution-def}
For given  $u_0(\cdot)\in C^0(\bar\Omega)$ and $v_0(\cdot)\in C^0(\bar\Omega)$ satisfying that
$u_0\ge 0$, $v_0\ge 0$, and $\int_\Omega (u_0(x)+v_0(x))dx>0$,
 we say $(u(t,x),v(t,x),w(t,x))$ is a {\rm classical solution} of \eqref{main-eq} on $(0,T)$ for some $T\in (0,\infty]$ with initial condition  $(u(0,x),v(0,x))=(u_0(x),v_0(x))$ if
\begin{equation*}
u(\cdot,\cdot),v(\cdot,\cdot)\in  C([0,T)\times\bar\Omega )\cap C^{1,2}(  (0,T)\times\bar\Omega),\quad
w(\cdot,\cdot)\in C^{0,2}((0,T)\times \bar\Omega),
\end{equation*}
\begin{equation*}
\lim_{t\to 0+}\|u(t,\cdot)-u_0(\cdot)\|_{C^0(\bar\Omega)}=0,\quad \lim_{t\to 0+}\|v(t,\cdot)-v_0(\cdot)\|_{C^0(\bar\Omega)}=0,
\end{equation*}
and $(u(t,x),v(t,x),w(t,x))$ satisfies \eqref{main-eq} for all $(t,x)\in (0,T)\times \Omega$.
\end{definition}

By quite standard arguments (for example, see \cite[Lemma 2.2]{FuWiYo1}),  the following proposition on the local existence of classical solutions of \eqref{main-eq} can be proved .

\begin{proposition} [Local existence]
\label{local-existence-prop}
{For given  $u_0(\cdot)\in C^0(\bar\Omega)$ and $v_0(\cdot)\in C^0(\bar\Omega)$ satisfying that
$u_0\ge 0$, $v_0\ge 0$, and $\int_\Omega (u_0(x)+v_0(x))dx>0$,  there exists $T_{\max}(u_0,v_0)\in (0,\infty]$
such that  \eqref{main-eq} has a unique  positive classical solution, denoted by $(u(t,x;u_0,v_0)$, $v(t,x;u_0,v_0)$, $w(t,x;u_0,v_0))$,  on $(0,T_{\max}(u_0,v_0))$ with initial condition $(u(0,x;u_0,v_0),v(0,x;u_0,v_0))=(u_0(x),v_0(x))$.
Moreover, if $T_{\max}(u_0,v_0)< \infty,$ then either }
\begin{equation*}
\limsup_{t \nearrow T_{\max}(u_0,v_0)} \left(\left\| u(t,\cdot;u_0,v_0)+v(t,\cdot;u_0,v_0) \right\|_{C^0(\bar \Omega)} \right) =\infty,
\end{equation*}
or
\begin{equation*}
    \liminf_{t \nearrow T_{\max}(u_0,v_0)} \inf_{x \in \Omega} w(t,x;u_0,v_0) =0.
\end{equation*}
\end{proposition}

{
\begin{proposition}
\label{upper-bound-prop} (\cite[Theorem 1.2(2)]{HKWS3})
For  any $q>2N$ and  $0<\theta<1-\frac{2N}{q}$,  there are $M_2>0$, $\beta>0$, and $\gamma>0$ such that for
 any   { $(u_0,v_0)$ satisfies \eqref{old-initial-cond-eq} and $T\in (0,\infty)$},  there holds
\begin{align}
\label{new-infinity-bdd-eq0}
\|u(t,\cdot;u_0,v_0)+v(t,\cdot;u_0,v_0)\|_{C^\theta(\bar\Omega)}
&\le  M_2 \Big[ (t-\tau)^{-\beta} e^{-\gamma (t-\tau)} \|u(\tau,\cdot;u_0,v_0)+v(\tau,\cdot;u_0,v_0)\|_{L^q}\nonumber\\
&\quad+\frac{\displaystyle \sup_{\tau \le t <\hat T}\|u(t,\cdot;u_0,v_0)+
v(t,\cdot;u_0,v_0)\|_{L^q}^2}{\displaystyle \inf_{\tau \le t< \hat T, x\in\Omega} w(t,x;u_0,v_0)}\nonumber\\
&\quad+\sup_{\tau \le t<\hat T}\|u(t,\cdot;u_0,v_0)+v(t,\cdot;u_0,v_0)\|_{L^q}+1\Big]
\end{align}
for any $0<\tau<t<\hat T=\min\{T,T_{\max}(u_0,v_0)\}$.
\end{proposition}}

\section{Lower bound of $w$ and proof of Theorem \ref{main-thm1}}

In this section, we investigate lower bounds of $w(t,x;u_0,v_0)$ and prove Theorem \ref{main-thm1}.

Throughout this section, for given $u_0,v_0$ satisfying   {\eqref{new-initial-cond-eq}},   if no confusion occurs, we may put
$$
(u(t,x),v(t,x),w(t,x)):=(u(t,x;u_0,v_0),v(t,x;u_0,v_0),w(t,x;u_0,v_0)),
$$
and we may also drop $(t,x)$ in $u(t,x)$ (resp. $v(t,x)$, $w(t,x)$).

In order to derive a positive lower bound for $w(t,x),$ it is essential to obtain a positive estimate from below for $\int_{\Omega} ( u(t,x)+v(t,x))dx.$ In this direction, we present the following lemmas.

\begin{lemma}
\label{lower-bound-lm1} For any $\tau\in [0,\infty)$,
    \begin{equation}
\label{new-u-upper-bound-eq}
    \int_{\Omega} u (t,x;u_0,v_0)dx \leq {\rm max}\Big\{\int_{\Omega} u(\tau,x;u_0,v_0)dx,  \frac{a_{1}|\Omega|}{b_{1}} \Big\} \quad\forall \,\,  t\ge \tau
\end{equation}
and
\begin{equation}
\label{new-v-upper-bound-eq}
    \int_{\Omega} v (t,x;u_0,v_0)dx \leq  {\rm max}\Big\{\int_{\Omega} v(\tau,x;u_0,v_0)dx,  \frac{a_{2}|\Omega|}{b_{2}} \Big\} \quad\forall \,\,  t\ge\tau.
\end{equation}
Moreover,
\begin{equation}
\label{new-u-v-upper-bd-eq}
\limsup_{t\to\infty}\int_\Omega u(t,x;u_0,v_0)dx\le \frac{a_1|\Omega|}{b_1}
\quad {\rm and}\quad \limsup_{t\to\infty} \int_\Omega v(t,x;u_0,v_0)dx\le \frac{a_2|\Omega|}{b_2}.
\end{equation}
\end{lemma}

\begin{proof}
It follows from \cite[Lemma 2.7]{HKWS3}.
\end{proof}

\begin{lemma}
\label{lower-bound-lm2}
For any  given $u_0,v_0$ satisfying   {\eqref{new-initial-cond-eq}},
\begin{equation*}
    w (t,x;u_0,v_0)\ge \delta_0  \min\{\nu,\lambda\}\int_{\Omega} (u(t,{ x};u_0,v_0) + v(t,{ x};u_0,v_0))d{ x} >0 \quad \forall\,\, t\ge 0,\,\, x\in\Omega,
\end{equation*}
where  $\delta_0$ is as in \eqref{delta-0-eq}.
\end{lemma}

\begin{proof}
It follows from the arguments of  \cite[Lemma 2.1]{FuWiYo} and the
Gaussian lower bound for the heat kernel of the laplacian with Neumann boundary
condition on smooth domain.
\end{proof}

\begin{lemma}
\label{lower-bound-lm3}
For any $p>0$ and  $u_0,v_0$ satisfying   {\eqref{new-initial-cond-eq}},
$$
\int_\Omega (u(t,x;u_0,v_0)+v(t,x;u_0,v_0))dx \ge |\Omega|^{\frac{p+1}{p}}\Big(\int_\Omega (u(t,x;u_0,v_0)+v(t,x;u_0,v_0))^{-p}dx\Big)^{-\frac{1}{p}}
$$
for all $t>0$.
\end{lemma}

\begin{proof}
Note that for any $t>0$,  we have
$$
u(t,x;u_0,v_0)+v(t,x;u_0,v_0)>0\quad \forall\, x\in\Omega.
$$
By H\"older  inequality, we have that for any $p>0$,
\begin{align*}
|\Omega|=\int_\Omega (u+v)^{\frac{p}{p+1}} (u+v)^{-\frac{p}{p+1}} dx\le \Big(\int_\Omega (u+v)dx\Big)^{\frac{p}{p+1}}\Big(\int_\Omega (u+v)^{-p}dx\Big)^{\frac{1}{p+1}}.
\end{align*}
The lemma then follows.
\end{proof}

\begin{lemma}
\label{lower-bound-lm4}
For any $p>0$,
\begin{equation*}
     p \int_{\Omega} \frac{(u+v)^{-p-1}}{w} \nabla (u+v) \cdot \nabla w + \int_{\Omega} (u+v)^{-p} \frac{|\nabla w|^2}{w^2} \leq \mu \int_{\Omega} (u+v)^{-p}
\end{equation*}
for all $t>0$.
\end{lemma}

\begin{proof}
Multiplying the third equation in \eqref{main-eq} by $\frac{(u+v)^{-p}}{w}$ and then integrating over $\Omega$ with respect to $x$, we obtain that
\begin{align*}
    0&= \int_{\Omega} \frac{(u+v)^{-p}}{w} \cdot \big( \Delta w-\mu w +\nu u+ \lambda v\big)\\
    &= - \int_{\Omega}\frac{(-p) (u+v)^{-p-1}w \nabla (u+v) -(u+v)^{-p} \nabla w}{w^2} \cdot \nabla w\\
    &\quad - \mu \int_{\Omega} (u+v)^{-p} + \nu   \int_{\Omega}\frac{u (u+v)^{-p}}{w} + \lambda  \int_{\Omega}\frac{v (u+v)^{-p}}{w}
\end{align*}
for all $t>0$.  We then have
\begin{equation*}
     p \int_{\Omega} \frac{(u+v)^{-p-1}}{w} \nabla (u+v) \cdot \nabla w + \int_{\Omega} (u+v)^{-p} \frac{|\nabla w|^2}{w^2} + \nu  \int_{\Omega}\frac{u(u+v)^{-p}}{w}+ \lambda  \int_{\Omega}\frac{v (u+v)^{-p}}{w}= \mu \int_{\Omega} (u+v)^{-p}
\end{equation*}
for all $t>0$.  The lemma is thus proved.
\end{proof}

\begin{lemma}
\label{lower-bound-lm5}
For any  $p>0$,  {there holds}
 \begin{align}
\label{chi-1-2-eq1}
& \int_\Omega (u+v)^{-p-2}\frac{\chi_1 u+\chi_2 v}{w} \nabla(u+v)\cdot\nabla w \nonumber\\
     &\le \frac{p(\chi_1-\chi_2)^2}{4\chi_2} \int_\Omega (u+v)^{-p-2} |\nabla (u+v)|^2 + \frac{\mu \chi_2}{p} \int_{\Omega} (u+v)^{-p} \quad \forall\, t>0,
\end{align}
and
\begin{align}
\label{chi-1-2-eq2}
& \int_\Omega (u+v)^{-p-2}\frac{\chi_1 u+\chi_2 v}{w} \nabla(u+v)\cdot\nabla w \nonumber\\
     &\le \frac{p(\chi_2-\chi_1)^2}{4\chi_1} \int_\Omega (u+v)^{-p-2} |\nabla (u+v)|^2 + \frac{\mu \chi_1}{p} \int_{\Omega} (u+v)^{-p} \quad \forall\, t>0.
\end{align}
Moreover, if   $\chi_1 = \chi_2:=\chi$, then for any $p>0$ and  $\beta>0$,
    \begin{align}
    \label{chi-1-2-eq3}
  &  \int_\Omega (u+v)^{-p-2}\frac{\chi_1 u+\chi_2 v}{w} \nabla(u+v)\cdot\nabla w \nonumber\\
 &\le   \frac{p(\chi-\beta)^2}{4\beta} \int_\Omega (u+v)^{-p-2} |\nabla (u+v)|^2 +
     \frac{\beta\mu}{p} \int_{\Omega} (u+v)^{-p}\quad \forall\, t>0.
\end{align}
\end{lemma}

\begin{proof}
First, note that
\begin{align*}
   &  \int_\Omega (u+v)^{-p-2}\frac{\chi_1 u+\chi_2 v}{w} \nabla(u+v)\cdot\nabla w \nonumber\\
&=\underbrace{\chi_2 \int_\Omega \frac{(u+v)^{-p-1}}{w}  \nabla(u+v)\cdot\nabla w}_{J_1} + \underbrace{(\chi_1-\chi_2) \int_\Omega \frac{(u+v)^{-p-2}}{w} u \nabla(u+v)\cdot\nabla w}_{J_2}.
\end{align*}
Applying Young's inequality on $J_2$, we get
\begin{align}
\label{eq-I-2-chi-1}
    (\chi_1-\chi_2)\int_\Omega \frac{(u+v)^{-p-2}}{w} u \nabla(u+v)\cdot\nabla w &\leq \frac{p(\chi_1-\chi_2)^2}{4\chi_2} \int_\Omega (u+v)^{-p-2} |\nabla (u+v)|^2\nonumber\\
    &\quad + \frac{\chi_2}{p} \int_\Omega (u+v)^{-p-2} u^2 \frac{|\nabla w|^2}{w^2}.
\end{align}
Since $u^2 \leq (u+v)^2,$ and by Lemma \ref{lower-bound-lm4}, we obtain that
\begin{align*}
   \int_\Omega (u+v)^{-p-2} u^2 \frac{|\nabla w|^2}{w^2} &\leq \int_\Omega (u+v)^{-p}  \frac{|\nabla w|^2}{w^2} \\
   & \le  \mu \int_{\Omega} (u+v)^{-p} - p \int_{\Omega} \frac{(u+v)^{-p-1}}{w} \nabla (u+v) \cdot \nabla w .
\end{align*}
This together  with \eqref{eq-I-2-chi-1} yields that
\begin{align*}
& \int_\Omega (u+v)^{-p-2}\frac{\chi_1 u+\chi_2 v}{w} \nabla(u+v)\cdot\nabla w \nonumber\\
& = \chi_2 \int_\Omega \frac{(u+v)^{-p-1}}{w}  \nabla(u+v)\cdot\nabla w  + (\chi_1-\chi_2) \int_\Omega \frac{(u+v)^{-p-2}}{w} u \nabla(u+v)\cdot\nabla w \nonumber\\
     & \leq \chi_2 \int_\Omega \frac{(u+v)^{-p-1}}{w}  \nabla(u+v)\cdot\nabla w  + \frac{p(\chi_1-\chi_2)^2}{4\chi_2} \int_\Omega (u+v)^{-p-2} |\nabla (u+v)|^2\nonumber
\\
     &\quad + \frac{\mu \chi_2}{p} \int_{\Omega} (u+v)^{-p} - \chi_2 \int_{\Omega} \frac{(u+v)^{-p-1}}{w} \nabla (u+v) \cdot \nabla w \nonumber\\
     &= \frac{p(\chi_1-\chi_2)^2}{4\chi_2} \int_\Omega (u+v)^{-p-2} |\nabla (u+v)|^2 + \frac{\mu \chi_2}{p} \int_{\Omega} (u+v)^{-p} .
\end{align*}
Therefore, \eqref{chi-1-2-eq1} holds.

Next, by the similar arguments as in the above, we have that
\begin{align*}
& \int_\Omega (u+v)^{-p-2}\frac{\chi_1 u+\chi_2 v}{w} \nabla(u+v)\cdot\nabla w \nonumber\\
& = \chi_1 \int_\Omega \frac{(u+v)^{-p-1}}{w}  \nabla(u+v)\cdot\nabla w  + (\chi_2-\chi_1) \int_\Omega \frac{(u+v)^{-p-2}}{w} u \nabla(u+v)\cdot\nabla w \nonumber\\
     & \leq \chi_1 \int_\Omega \frac{(u+v)^{-p-1}}{w}  \nabla(u+v)\cdot\nabla w  + \frac{p(\chi_2-\chi_1)^2}{{ 4\chi_1}} \int_\Omega (u+v)^{-p-2} |\nabla (u+v)|^2
\nonumber\\
     &\quad + \frac{\mu { \chi_1}}{p} \int_{\Omega} (u+v)^{-p} - { \chi_1} \int_{\Omega} \frac{(u+v)^{-p-1}}{w} \nabla (u+v) \cdot \nabla w \nonumber\\
     &= \frac{p(\chi_2-\chi_1)^2}{4\chi_1} \int_\Omega (u+v)^{-p-2} |\nabla (u+v)|^2 + \frac{\mu \chi_1}{p} \int_{\Omega} (u+v)^{-p} .
\end{align*}
Therefore, \eqref{chi-1-2-eq2} holds.

Now, assume that  $\chi_1=\chi_2$. Let $\chi:=\chi_1$. For any $\beta>0$ and $p>0$,
we have
\begin{align*}
 &   \chi \int_\Omega \frac{(u+v)^{-p-1}}{w} \nabla(u+v)\cdot\nabla w \nonumber\\
&= \underbrace{(\chi-\beta) \int_\Omega \frac{(u+v)^{-p-1}}{w} \nabla(u+v)\cdot\nabla w}_{I_1}+\underbrace{\beta \int_\Omega \frac{(u+v)^{-p-1}}{w} \nabla(u+v)\cdot\nabla w}_{I_2}.
\end{align*}
By Young's inequality, we get that
\begin{align}
      \label{eq-I-1}
    I_1&=(\chi-\beta) \int_\Omega \frac{(u+v)^{-p-1}}{w} \nabla(u+v)\cdot\nabla w
\nonumber\\
&\le  \frac{p(\chi-\beta)^2}{4\beta} \int_\Omega (u+v)^{-p-2} |\nabla (u+v)|^2+ \frac{ \beta }{p} \int_{\Omega} (u+v)^{-p} \frac{|\nabla w|^2}{w^2}.
\end{align}
By Lemma \ref{lower-bound-lm4}, we get that
\begin{equation}
    \label{eq-I-2}
    I_2=\beta \int_\Omega \frac{(u+v)^{-p-1}}{w} \nabla(u+v)\cdot\nabla w \leq \frac{\beta\mu}{p} \int_{\Omega} (u+v)^{-p} - \frac{\beta}{p} \int_{\Omega} (u+v)^{-p} \frac{|\nabla w|^2}{w^2}.
\end{equation}
Therefore \eqref{eq-I-1} and \eqref{eq-I-2} yield \eqref{chi-1-2-eq3}.
The lemma is thus proved.
\end{proof}

We now prove Theorem \ref{main-thm1}.

\begin{proof}[Proof of Theorem \ref{main-thm1}(1)]
First of all, by adding the first two equations in \eqref{main-eq}, we have
\begin{align}
\label{lower-bound-proof-eq00}
{ \frac{\p}{\p t}}(u+v)=& \Delta(u+v)-\nabla\cdot \left(\frac{\chi_1 u+\chi_2 v}{w}\nabla w\right)\nonumber\\
& +(a_1u+a_2 v)-(b_1u^2+c_2 v^2)-(c_1+b_2)u v .
\end{align}
For any $p>0$,
multiplying the above equation by  $(u+v)^{-p-1}$ and integrating over $\Omega$, we get
\begin{align}
\label{lower-bound-proof-eq0}
\frac{1}{p}\frac{d}{dt}\int_\Omega (u+v)^{-p}=& -(p+1)\int_\Omega (u+v)^{-p-2}|\nabla (u+v)|^2\nonumber\\
& +(p+1)\int_\Omega (u+v)^{-p-2}\frac{\chi_1 u+\chi_2 v}{w} \nabla(u+v)\cdot\nabla w\nonumber\\
& -\int_\Omega (u+v)^{-p-1}(a_1u+a_2 v)+\int_\Omega (u+v)^{-p-1}(b_1 u^2+c_2 v^2)\nonumber\\
&+\int_\Omega (u+v)^{-p-1}(c_1+b_2) uv\nonumber\\
\le& - (p+1)\int_\Omega (u+v)^{-p-2}|\nabla (u+v)|^2\nonumber\\
&+(p+1)\int_\Omega (u+v)^{-p-2}\frac{\chi_1 u+\chi_2 v}{w} \nabla(u+v)\cdot\nabla w\nonumber\\
& - a_{\min} \int_\Omega (u+v)^{-p}+b_{\max} \int_\Omega (u+v)^{-p+1}+c_{\max} \int_\Omega (u+v)^{-p+1}.
\end{align}

The rest of the proof is divided into three steps.

\smallskip

\noindent {\bf Step 1.} In this step,   we prove that there is $0<p(\chi_1,\chi_2)\le 1$ such that
\begin{align}
\label{lower-bound-proof-eq1}
\frac{1}{p}\frac{d}{dt}\int_\Omega (u+v)^{-p}\le & - \Big(a_{\min}-{ 2\mu}h(\chi_1,\chi_2)\Big) \int_\Omega (u+v)^{-p}\nonumber\\
&+b_{\max} \int_\Omega (u+v)^{-p+1}+c_{\max} \int_\Omega (u+v)^{-p+1},
\end{align}
where $h(\chi_1,\chi_2)$ is as in \eqref{h-function-eq}.

We divide the proof of \eqref{lower-bound-proof-eq1} into three  cases.

\smallskip

\noindent {\bf Case 1.}  {\bf $\frac{(\chi_1-\chi_2)^2}{4\chi_2}\le 1$. }
In this case, let  $p=1$. Then by \eqref{chi-1-2-eq1} and \eqref{lower-bound-proof-eq0}, we have
\begin{align*}
\frac{1}{p} \frac{d}{dt}\int_\Omega (u+v)^{-p}\le &  - (p+1)\int_\Omega (u+v)^{-p-2}|\nabla (u+v)|^2\nonumber\\
&+(p+1)\int_\Omega (u+v)^{-p-2}\frac{\chi_1 u+\chi_2 v}{w} \nabla(u+v)\cdot\nabla w\nonumber\\
& - a_{\min} \int_\Omega (u+v)^{-p}+b_{\max} \int_\Omega (u+v)^{-p+1}+c_{\max} \int_\Omega (u+v)^{-p+1}\nonumber\\
\le & - (p+1)\int_\Omega (u+v)^{-p-2}|\nabla (u+v)|^2\nonumber\\
&+(p+1)\Big( \frac{p(\chi_2-\chi_1)^2}{4\chi_2} \int_\Omega (u+v)^{-p-2} |\nabla (u+v)|^2 + \frac{\mu \chi_1}{p} \int_{\Omega} (u+v)^{-p} \Big)\nonumber\\
& - a_{\min} \int_\Omega (u+v)^{-p}+b_{\max} \int_\Omega (u+v)^{-p+1}+c_{\max} \int_\Omega (u+v)^{-p+1}\nonumber\\
= & -\Big(a_{\min} -2\mu \chi_1\Big)\int_\Omega (u+v)^{-p}+b_{\max} \int_\Omega (u+v)^{-p+1}+c_{\max} \int_\Omega (u+v)^{-p+1}.
\end{align*}
Therefore, \eqref{lower-bound-proof-eq1} holds with $p=1$.

\smallskip

\noindent {\bf Case 2.}  {\bf $\frac{(\chi_1-\chi_2)^2}{4\chi_1}\le 1$. } In this case,  let $p=1$. By the similar arguments as in {\bf Case 1}, we have that
\begin{align*}
\frac{1}{p} \frac{d}{dt}\int_\Omega (u+v)^{-p} \le &  -\Big(a_{\min} -2\mu \chi_2\Big)\int_\Omega (u+v)^{-p}\nonumber\\
&+b_{\max} \int_\Omega (u+v)^{-p+1}+c_{\max} \int_\Omega (u+v)^{-p+1}.
\end{align*}
Therefore, \eqref{lower-bound-proof-eq1} holds with $p=1$.

\smallskip

\noindent {\bf Case 3.} {\bf $\frac{(\chi_1-\chi_2)^2}{4\chi_2}> 1$} and  {\bf $\frac{(\chi_1-\chi_2)^2}{4\chi_1}> 1$. } In this case, let
$$
p=\frac{4\chi_2}{(\chi_1-\chi_2)^2}.
$$
  Then by \eqref{chi-1-2-eq1} and \eqref{lower-bound-proof-eq0}, we have
\begin{align*}
&\frac{1}{p} \frac{d}{dt}\int_\Omega (u+v)^{-p}\nonumber\\
 \le&   - (p+1)\int_\Omega (u+v)^{-p-2}|\nabla (u+v)|^2\nonumber\\
&+(p+1)\Big( \frac{p(\chi_2-\chi_1)^2}{4\chi_2} \int_\Omega (u+v)^{-p-2} |\nabla (u+v)|^2 + \frac{\mu \chi_2}{p} \int_{\Omega} (u+v)^{-p} \Big)\nonumber\\
& - a_{\min} \int_\Omega (u+v)^{-p}+b_{\max} \int_\Omega (u+v)^{-p+1}+c_{\max} \int_\Omega (u+v)^{-p+1}\nonumber\\
\le  & -\Big(a_{\min} -\frac{\mu(\chi_1-\chi_2)^2}{2}\Big)\int_\Omega (u+v)^{-p}+b_{\max} \int_\Omega (u+v)^{-p+1}+c_{\max} \int_\Omega (u+v)^{-p+1}.
\end{align*}
Therefore, \eqref{lower-bound-proof-eq1} holds with $p=\frac{4\chi_2}{(\chi_1-\chi_2)^2}$.
The proof of \eqref{lower-bound-proof-eq1} is thus completed.

\smallskip

\smallskip
\noindent {\bf Step 2.} In this step, we prove that
\begin{equation}
\label{lower-bound-proof-eq2}
\limsup_{t\to\infty} \int_\Omega (u+v)^{-p}\le  \frac{m_0 |\Omega|}{a_{\min}-{ 2}h(\chi_1,\chi_2)},
\end{equation}
where $h(\chi_1,\chi_2)$ is as in \eqref{h-function-eq},  $m_0$ is as in \eqref{m-0-eq},
and  $p$ is as in {\bf Step 1.}

First,  by H\"older's  inequality, we have
\begin{align*}
    b_{\max} \int_\Omega (u+v)^{-p+1} \leq&  b_{\max} |\Omega|^p \cdot \left(\int_{\Omega} u+v\right)^{1-p},\\
    c_{\max} \int_\Omega (u+v)^{-p+1} \leq& c_{\max} |\Omega|^p \cdot \left(\int_{\Omega} u+v\right)^{1-p}.
\end{align*}
Then by Lemma \ref{lower-bound-lm1}, we have
\begin{align*}
    b_{\max}\limsup_{t\to\infty}  \int_\Omega (u+v)^{-p+1} \leq&  b_{\max} |\Omega| \cdot \left(\frac{a_1}{b_1}+\frac{a_2}{b_2}\right)^{1-p}\le   b_{\max} |\Omega| \cdot \max\Big\{1, \frac{a_1}{b_1}+\frac{a_2}{b_2}\Big\},\\
    c_{\max} \limsup_{t\to\infty} \int_\Omega (u+v)^{-p+1} \leq& c_{\max} |\Omega|  \cdot \left(\frac{a_1}{b_1}+\frac{a_2}{b_2}\right)^{1-p}\le  c_{\max} |\Omega|  \cdot \max\{1,\frac{a_1}{b_1}+\frac{a_2}{b_2}\Big\}.
\end{align*}
This together with \eqref{lower-bound-proof-eq1} implies \eqref{lower-bound-proof-eq2}.

\smallskip

\noindent {\bf Step 3.}  In this step, we prove \eqref{new-lower-bound-eq1} and
\eqref{lower-bound-proof-eq3}.

By Lemma  \ref{lower-bound-lm3}, we have
$$
\int_\Omega (u(t,x)+v(t,x))dx \ge |\Omega|^{\frac{p+1}{p}}\Big(\int_\Omega (u(t,x)+v(t,x))^{-p}dx\Big)^{-\frac{1}{p}}
$$
This together with \eqref{lower-bound-proof-eq2} implies \eqref{new-lower-bound-eq1}, that is, 
$$
\liminf_{t\to\infty}\int_\Omega (u+v)\ge  \frac{|\Omega|\Big (a_{\min}-{ 2}h(\chi_1,\chi_2)\Big)^{1/p}}{m_0^{1/p}},
$$
which together with Lemma \ref{lower-bound-lm2} implies \eqref{lower-bound-proof-eq3}.
Theorem \ref{main-thm1}(1) is thus proved.
\end{proof}

\begin{proof}[Proof of Theorem \ref{main-thm1}(2)] Assume $\chi_1=\chi_2$. Let $\chi=\chi_1$. By \eqref{chi-1-2-eq3} with $p=1$,   for  $\beta>0$, we have
    \begin{align*}
  &  \int_\Omega (u+v)^{-3}\frac{\chi_1 u+\chi_2 v}{w} \nabla(u+v)\cdot\nabla w \nonumber\\
 &\le   \frac{(\chi-\beta)^2}{4\beta} \int_\Omega (u+v)^{-3} |\nabla (u+v)|^2 +
     {\beta\mu}\int_{\Omega} (u+v)^{-1}\quad \forall\, t>0.
\end{align*}
By \eqref{lower-bound-proof-eq0} with $p=1$, we have
\begin{align*}
\frac{d}{dt}\int_\Omega (u+v)^{-1}
\le& - 2\int_\Omega (u+v)^{-p-2}|\nabla (u+v)|^2\nonumber\\
&+2      \frac{(\chi-\beta)^2}{4\beta} \int_\Omega (u+v)^{-3} |\nabla (u+v)|^2 +
   2  {\beta\mu}\int_{\Omega} (u+v)^{-1}  \nonumber\\
& - a_{\min} \int_\Omega (u+v)^{-p}+b_{\max} \int_\Omega (u+v)^{-p+1}+c_{\max} \int_\Omega (u+v)^{-p+1}.
\end{align*}
Let $\beta$ be such that
$$
(\chi-\beta)^2=4\beta.
$$
Then
$$
\beta=\chi+2-2\sqrt{\chi+1},
$$
and
$$
2\beta =2\big(\sqrt {\chi+1}-1\big)^2.
$$
 We then have
\begin{align*}
\frac{d}{dt}\int_\Omega (u+v)^{-1}
\le  -\big(a_{\min}-2\mu {  \big(\sqrt{\chi+1}-1\big)^2}\big) \int_{\Omega} (u+v)^{-1}+(b_{\max}+c_{\max})|\Omega|.
\end{align*}
This together with the comparison principle for scalar ODEs  implies that
$$
\limsup_{t\to\infty} \int_\Omega (u+v)^{-1}\le \frac{(b_{\max}+c_{\max})|\Omega|}
{a_{\min}-2\mu { \big(\sqrt{\chi+1}-1\big)^2}}.
$$
Then by Lemma  \ref{lower-bound-lm3}, we have
$$
\liminf_{t\to\infty} \int_\Omega (u+v)\ge |\Omega|\frac{a_{\min}-2\mu{ (\sqrt{\chi+1}-1)^2}}{b_{\max}+c_{\max}}.
$$
Hence \eqref{new-lower-bound-eq2} holds.
By Lemma \ref{lower-bound-lm2}, we obtain
$$
\liminf_{t\to\infty}\inf_{x\in\Omega} w(t,x)\ge \delta_0\cdot\min\{\nu,\lambda\}\cdot  |\Omega| \cdot \frac{a_{\min}-2\mu{ (\sqrt{\chi+1}-1)^2}}{b_{\max}+c_{\max}}.
$$
Theorem \ref{main-thm1}(2) is thus proved.
\end{proof}

\section{Stabilization and proof of Theorem \ref{main-thm2}}

In this section, we assume   \eqref{weak-competition-eq1}. We  study the stabilization in \eqref{main-eq}, i.e., the global stability of
the positive constant solution$(u^*,v^*,w^*)$ of \eqref{main-eq}, and prove
Theorem \ref{main-thm2}. { Throughout this section, we assume that $(u_0,v_0)$ satisfies \eqref{new-initial-cond-eq}.}

Here is  the main idea of the proof of Theorem \ref{main-thm2}.
We prove Theorem \ref{main-thm2} by first showing that
\begin{equation}
\label{stability-eq0-0}
\|u(t,\cdot;u_0,v_0)-u^*\|_{L^2}+\|v(t,\cdot;u_0,v_0)-v^*\|_{L^2}\to 0
\end{equation}
as $t\to\infty$, and then showing that
 \begin{equation}
\label{stability-eq0-1}
\|u(t,\cdot;u_0,v_0)-u^*\|_{L^\infty}+\|v(t,\cdot;u_0,v_0)-v^*\|_{L^\infty}\to 0
\end{equation}
as $t\to\infty$.
We prove \eqref{stability-eq0-0} by showing that there is $\epsilon>0$ such that
\begin{equation*}
\epsilon \int_\tau^ t\int_\Omega (u(s,x;u_0,v_0)-u^*)^2+(v(s,x;u_0,v_0)-v^*)^2\le {E(\tau)} \quad \forall\,\, t>\tau\gg 1,
\end{equation*}
 where $E(\tau)=E(\tau;u_0,v_0)$ is as in \eqref{energy-function-eq}.

Before proving Theorem \ref{main-thm2}, we prove some lemmas.
As in section 3,  for any given $u_0,v_0$ satisfying {\eqref{new-initial-cond-eq}}, if no confusion occurs, we may put
$$
(u(t,x),v(t,x),w(t,x)):=(u(t,x;u_0,v_0),v(t,x;u_0,v_0),w(t,x;u_0,v_0)),
$$
and  we may also drop $(t,x)$ in $u(t,x)$ (resp. $v(t,x)$, $w(t,x)$).
Throughout this section, we assume that \eqref{cond-on-chi-eq1} holds.

\begin{lemma}
\label{stability-lm0-1}
Suppose that $f(t): (\tau, \infty) \rightarrow \mathbb{R}$ is a uniformly continuous nonnegative function such that $\int_{\tau}^{\infty} f(t)dt < \infty$. Then $f(t) \rightarrow 0$ as $t \rightarrow \infty.$

\end{lemma}

\begin{proof}
It follows from the arguments of  \cite[Lemma 3.1]{ba-wi}.
\end{proof}

\begin{lemma}
\label{stability-lm0-2}
Let
$$
f(t)=\int_\Omega \Big[(u(t,x;u_0,v_0)-u^*)^2+(v(t,x;u_0,v_0)-v^*)^2\Big]dx.
$$
Then $f(t)$ is uniformly continuous in $t\ge \tau>0$ for any $\tau>0$.
\end{lemma}

\begin{proof}
It suffices to prove that $f^{'}(t)$ is bounded on $[\tau,\infty)$ for any $\tau>0$.
By a direct calculation, we have
\begin{align*}
f^{'}(t)&=2\int_\Omega(u-u^*) u_t+2\int_\Omega (v-v^*) v_t\nonumber\\
&=2\int_\Omega (u-u^*) \Big(\Delta u-\chi_1\nabla \big(\frac{u}{w}\nabla w\big)+u(a_1-b_1u-c_1 v)\Big)\nonumber\\
&\quad +2\int_\Omega (v-v^*)\Big(\Delta v-\chi_2 \nabla \cdot\big(\frac{v}{w}\nabla v)+ v(a_2-b_2v-c_2u)\Big)\nonumber\\
&=-2\int_\Omega |\nabla u|^2+{ 2 \chi_1} \int_\Omega \frac{u}{w}\nabla u\cdot\nabla w+
{ 2 \int_\Omega (u-u^*) u}(a_1-b_1u-c_2v)\nonumber\\
&\quad -2\int_\Omega |\nabla v|^2+{ 2 \chi_2}\int_\Omega \frac{v}{w}\nabla v\cdot\nabla w+{2 \int_\Omega (v-v^*) v}(a_2-b_2v-c_2u).
\end{align*}
{Note that
$$
\begin{cases}
\Delta w-\mu w+\nu u+\lambda v=0,\quad &x\in\Omega\cr
\frac{\p w}{\p n}=0,\quad &x\in\p\Omega.
\end{cases}
$$
{By \eqref{new-initial-cond-eq}, for any $q>2N$,
$$
\limsup_{t\to\infty}\|u(t,\cdot;u_0,v_0)+v(t,\cdot;u_0,v_0)\|_{L^q}<\infty.
$$
By Theorem \ref{main-thm1},
\begin{equation}
\label{new-eq2}
\inf_{t\in [\tau,\infty),x\in\Omega} w(t,x;u_0,v_0)>0\quad \forall\,\tau>0.
\end{equation}
Then by Proposition \ref{upper-bound-prop},}  and a priori estimates for elliptic equations, we have
\begin{equation}
\label{new-eq1}
\sup_{t\ge \tau} \Big\{\|w(t,\cdot)\|_{C^{\theta}(\bar \Omega)},\|\nabla w(t,\cdot)\|_{C^\theta(\bar\Omega)},\|\Delta w(t,\cdot)\|_{C^\theta(\bar\Omega)}\Big\}<\infty
\end{equation}
 for some $0<\theta<1$ and any $\tau>0$.
Note that
\begin{equation*}
\begin{cases}
u_t=\Delta u+{\bf q(t,x)} \cdot \nabla u+p(t,x)u,\quad &x\in \Omega\cr
\frac{\p u}{\p n}=0,\quad &x\in\p\Omega,
\end{cases}
\end{equation*}
where
$$
{\bf q(t,x)}=-\chi_1 \frac{\nabla w}{w},\quad p(t,x)=-\chi_1 \frac{\Delta w}{w} +\chi_1 \frac{|\nabla w|^2}{w^2}+(a_1-b_1u-c_2 v).
$$
It then follows from  \eqref{new-eq2},   \eqref{new-eq1},   Proposition \ref{upper-bound-prop}, and a priori estimates for parabolic equations  that
\begin{equation}
\label{new-eq3}
\sup_{t\ge \tau}\|u(t,\cdot)\|_{C^1(\bar \Omega)}<\infty\quad \forall\,\tau>0.
\end{equation}
Similarly, we have
\begin{equation}
\label{new-eq4}
\sup_{t\ge \tau}\|v(t,\cdot)\|_{C^1(\bar \Omega)}<\infty\quad \forall\,\tau>0.
\end{equation}
By \eqref{new-eq1}-\eqref{new-eq4},}
  $f^{'}(t)$ is bounded on $[\tau,\infty)$ for any $\tau>0$.
The lemma is thus proved.
\end{proof}

\begin{lemma}
\label{stability-lm1}
Assume \eqref{weak-competition-eq1}. Let $(u^*,v^*,w^*)$ be as in \eqref{positive-constant-solu-eq}.
Then
$$
\int_\Omega(w-w^*)^2\le \frac{2\nu^2}{\mu^2}\int_\Omega(u-u^*)^2+\frac{2\lambda^2}{\mu^2}\int_\Omega (v-v^*)^2.
$$
\end{lemma}

\begin{proof}
First, note that
$$
-\Delta(w-w^*)+\mu(w-w^*)=\nu(u-u^*)+\lambda (v-v^*).
$$
Multiplying the above equation by $(w-w^*)$ and integrating over $\Omega$, we get
\begin{align*}
&\int_\Omega |\nabla (w-w^*)|^2+\mu\int_\Omega (w-w^*)^2\nonumber\\
&=\nu\int_\Omega(w-w^*)(u-u^*)+\lambda\int_\Omega (w-w^*)(v-v^*)\nonumber\\
&\le \frac{\mu}{4}\int_\Omega (w-w^*)^2+\frac{\nu^2}{\mu}\int_\Omega(u-u^*)^2+
\frac{\mu}{4}\int_\Omega (w-w^*)^2+\frac{\lambda^2}{\mu}\int_\Omega(v-v^*)^2.
\end{align*}
It then follows that
$$
\int_\Omega(w-w^*)^2\le \frac{2\nu^2}{\mu^2}\int_\Omega(u-u^*)^2+\frac{2\lambda^2}{\mu^2}\int_\Omega (v-v^*)^2.
$$
The lemma is thus proved.
\end{proof}

\begin{lemma}
\label{stability-lm2}
Assume \eqref{weak-competition-eq1}. Let $(u^*,v^*,w^*)$ be as in \eqref{positive-constant-solu-eq}. Then
$$
\int_\Omega \frac{|\nabla w|^2}{w^2}\le \frac{m^*(\mu,\nu,\lambda)}{w^*} \frac{1}{\inf_{x\in\Omega} w(t,x)} \Big[\int_\Omega (u-u^*)^2
+\int_\Omega (v-v^*)^2\Big],
$$
where  $m^*(\mu,\nu,\lambda)$ is as in \eqref{m-star-eq}.
\end{lemma}

\begin{proof}
First,   note that $\mu w^*=\nu u^*+\lambda v^*$.  Multiplying the third equation in \eqref{main-eq} by $\frac{w(t,x)-w^*}{w(t,x)}$ and then  integrating over $\Omega$ with respect to $x,$ we obtain that
\begin{align*}
    0&= \int_{\Omega} \left(\frac{w(t,x)-w^*}{w(t,x)}\right) \cdot \Big( \Delta w(t,x) -\mu w(t,x) + \nu u(t,x)+\lambda v{ (t,x)}\Big)dx\nonumber\\
&=-\int_\Omega w^* \frac{|\nabla w|^2}{w^2}-\mu \int_\Omega (w-w^*)+\nu
\int_\Omega \frac{w-w^*}{w} u+\lambda \int\frac{w-w^*}{w} v\nonumber\\
&=-\int_\Omega w^* \frac{|\nabla w|^2}{w^2}-\mu \int_\Omega\Big[ \frac{1}{w}(w-w^*)^2+w^*-\frac{(w^*)^2}{w}\Big]+
\int_\Omega \frac{w-w^*}{w}(\nu  u+\lambda v)\nonumber\\
&=-\int_\Omega w^* \frac{|\nabla w|^2}{w^2}-\mu \int_\Omega  \frac{1}{w}(w-w^*)^2-\mu \int_\Omega \frac{w^*}{w}(w-w^*)+
\int_\Omega \frac{w-w^*}{w}(\nu  u+\lambda v)\nonumber\\
&=-\int_\Omega w^* \frac{|\nabla w|^2}{w^2}-\mu \int_\Omega  \frac{1}{w}(w-w^*)^2
+ \int_\Omega \frac{w-w^*}{w}\Big(-\mu w^*+\nu u+\lambda v\Big)
\nonumber\\
&= -\int_\Omega w^* \frac{|\nabla w|^2}{w^2}-\mu \int_\Omega  \frac{1}{w}(w-w^*)^2
+ \int_\Omega \frac{w-w^*}{w}\Big(\nu (u-u^*)+\lambda (v-v^*)\Big).
\end{align*}
This together with Young's inequality  implies that
\begin{align*}
\int_\Omega \frac{|\nabla w|^2}{w^2}&=\frac{1}{w^*} \Big[ \int_\Omega \frac{w-w^*}{w} \big( -\mu (w-w^*)
+ \nu (u-u^*)+\lambda (v-v^*)\big)\Big]\nonumber\\
&\le \frac{1}{w^*} \int _\Omega\frac{1}{w}
\Big[(-\mu+\frac{\nu}{2}+\frac{\lambda}{2})  (w-w^*)^2+\frac{\nu}{2} (u-u^*)^2+
\frac{\lambda}{2}(v-v^*)^2\Big].
\end{align*}
By Lemma \ref{stability-lm1}, we have
\begin{align*}
\int_\Omega \frac{|\nabla w|^2}{w^2}\le & \frac{1}{w^*} \frac{1}{\inf_{x\in\Omega} w(t,x)} \Big(\max\Big\{0, -\mu+\frac{\nu}{2}+\frac{\lambda}{2}\Big\} \frac{2\nu^2}{\mu^2} +\frac{\nu}{2}\Big)\int_\Omega (u-u^*)^2\nonumber\\
&
+  \frac{1}{w^*} \frac{1}{\inf_{x\in\Omega} w(t,x)}  \Big(\max\Big\{1,-\mu+\frac{\nu}{2}+\frac{\lambda}{2}\Big\}\frac{2\lambda ^2}{\mu^2} +\frac{\lambda}{2}\Big)\int_\Omega (v-v^*)^2.
\end{align*}
The lemma then follows.
\end{proof}

\begin{lemma}
\label{stability-lm3}
Assume \eqref{cond-on-chi-eq2}. There is $\epsilon>0$ such that
\begin{equation}
\label{E-eq1}
E^{'}(t)\le -\epsilon \Big(\int_\Omega(u-u^*)^2+\int_\Omega (v-v^*)^2\Big)\quad \forall\,\, t\gg 1,
\end{equation}
where  $E(t)=E(t;u_0,v_0)$ is as in \eqref{energy-function-eq}.
\end{lemma}

\begin{proof}
First of all, recall that
\begin{align*}
E(t)=E(t;u_0,v_0)&=\xi_1\int_\Omega \Big(u(t,x;u_0,v_0)-u^*-u^*\ln\frac{u(t,x;u_0,v_0)}{u^*}\Big)\nonumber\\
&\quad +\xi_2\int_\Omega(v(t,x;u_0,v_0)-v^*-v^*\ln\frac{v(t,x;u_0,v_0)}{v^*}\Big).
\end{align*}
We then have
\begin{align}
\label{stability-eq1}
E^{'}(t)&=\xi_1\int_\Omega \Big(1-\frac{u^*}{u}\Big)u_t+\xi_2\int_\Omega\Big(1-\frac{v^*}{v}\Big) v_t\nonumber\\
&=\xi_1\int_\Omega \frac{u-u^*}{u}\Big( \Delta u-\chi_1\nabla\cdot \big(\frac{u}{w}\nabla w\big)\Big)+{ \xi_1}\int_\Omega (u-u^*)(a_1-b_1 u-c_1 v)\nonumber\\
&\quad
+\xi_2\int_\Omega\frac{v-v^*}{v}\Big(\Delta v-\chi_2\nabla \cdot\big(\frac{v}{w}\nabla w\big)\Big)+\xi_2\int_\Omega (v-v^*)(a_2-b_2v-c_2 u).
\end{align}
Note that
$$
a_1=b_1u^*+c_1v^*\quad {\rm and}\quad a_2=b_2v^*+c_2 u^*.
$$
This together with \eqref{stability-eq1} implies that
\begin{align}
\label{stability-eq2}
E^{'}(t)&=\xi_1 \int_\Omega \frac{u-u^*}{u}\Big( \Delta u-\chi_1\nabla\cdot \big(\frac{u}{w}\nabla w\big)\Big)-\xi_1 b_1 \int_\Omega (u-u^*)^2-\xi_1 c_1 \int_\Omega (u-u^*)(v-v^*)\nonumber\\
&\quad
+\xi_2 \int_\Omega\frac{v-v^*}{v}\Big(\Delta v-\chi_2\nabla \cdot\big(\frac{v}{w}\nabla w\big)\Big)-\xi_2 b_2\int_\Omega (v-v^*)^2-\xi_2 c_2\int_\Omega (u-u^*)(v-v^*)\nonumber\\
&=-\xi_1u^* \int_\Omega u^{-2}|\nabla u|^2-\xi_2 v^*\int_\Omega v^{-2}|\nabla v|^2+\xi_1 \chi_1 u^*\int_\Omega \frac{1}{uw} \nabla u\cdot\nabla w +\xi_2 \chi_2 v^*\int_\Omega\frac{1}{vw}\nabla v\cdot\nabla w\nonumber\\
&\quad -\xi_1 b_1\int_\Omega (u-u^*)^2-\xi_2b_2\int_\Omega (v-v^*)^2-(\xi_1 c_1+\xi_2 c_2)\int_\Omega
(u-u^*)(v-v^*).
\end{align}

Next, note that
\begin{equation}
\label{stability-eq3}
\xi_1 \chi_1 u^*\int_\Omega\frac{1}{uw}\nabla u\cdot\nabla w\le \xi_1 u^*\int_\Omega u^{-2}|\nabla u|^2+\frac{\xi_1 \chi_1^2 u^*}{4}\int_\Omega\frac{|\nabla w|^2}{w^2}
\end{equation}
and
\begin{equation}
\label{stability-eq4}
\xi_2 \chi_2 v^*\int_\Omega\frac{1}{vw}\nabla v\cdot\nabla w\le \xi_2 v^*\int_\Omega v^{-2}|\nabla v|^2+\frac{\xi_2 \chi_2^2 v^*}{4}\int_\Omega\frac{|\nabla w|^2}{w^2}.
\end{equation}
{By \eqref{xi-1-2-eq1},} we have
$0<\xi_i\le 1$ for $i=1,2$. Then
by \eqref{stability-eq2}, \eqref{stability-eq3}, \eqref{stability-eq4},  and  Lemma \ref{stability-lm2}
we obtain
\begin{align}
\label{stability-eq5}
E^{'}(t)&\le \Big(\frac{\xi_1 \chi_1^2 u^*}{4}+\frac{\xi_2 \chi_2^2 v^*}{4}\Big)\int_\Omega \frac{|\nabla w|^2}{w^2}\nonumber\\
&\quad -\xi_1b_1\int_\Omega (u-u^*)^2-\xi_2b_2\int_\Omega (v-v^*)^2-(\xi_1c_1+\xi_2 c_2)\int_\Omega (u-u^*)(v-v^*)\nonumber\\
 &\le \Big(\frac{\chi_1^2 u^*}{4}+\frac{\chi_2^2 v^*}{4}\Big)  \frac{m^*(\mu,\nu,\lambda)}{w^*} \frac{1}{\inf_{t\ge \tau, x\in\Omega} w(t,x)} \Big[\int_\Omega (u-u^*)^2
+\int_\Omega (v-v^*)^2\Big]\nonumber\\
&\quad -\xi_1b_1\int_\Omega (u-u^*)^2-\xi_2b_2\int_\Omega (v-v^*)^2-(\xi_1c_1+\xi_2 c_2)\int_\Omega (u-u^*)(v-v^*)\nonumber\\
&=-\int_\Omega (u-u^*,v-v^*) A_\tau (u-u^*,v-v^*)^\top
\end{align}
for any $t\ge\tau>0$,
where $A_\tau$ is the following $2\times 2 $ symmetric matrix,
$$
A_\tau =\left(\begin{matrix}
\xi_1b_1-   \Big(\frac{\chi_1^2 u^*}{4}+\frac{\chi_2^2 v^*}{4}\Big)  \frac{m^*(\mu,\nu,\lambda)}{w^* \inf_{t\ge \tau,x\in\Omega} w(t,x)} & \frac{1}{2}\big(\xi_1 c_1+\xi_2 c_2\big)   \\
&\\
  \frac{1}{2}\big(\xi_1 c_1+\xi_2 c_2\big)   & \xi_2 b_2- \Big(\frac{\chi_1^2 u^*}{4}+\frac{\chi_2^2 v^*}{4}\Big)  \frac{m^*(\mu,\nu,\lambda)}{w^* \inf_{t\ge\tau,x\in\Omega} w(t,x) }
\end{matrix}\right).
$$
{Note that
$$
A_{\tau}=B(\xi_1,\xi_2,\eta)
$$
with
$$
\eta=\Big(\frac{\chi_1^2 u^*}{4}+\frac{\chi_2^2 v^*}{4}\Big)  \frac{m^*(\mu,\nu,\lambda)}{w^* \inf_{t\ge\tau,x\in\Omega} w(t,x) }.
$$  }

 Now, by \eqref{cond-on-chi-eq2},
we have
\begin{align*}
a_{\min}-2\mu h(\chi_1,\chi_2)&>\Big(\frac{\chi_1^2 u^*+\chi_2^2 v^*}{4}\Big)^p \cdot m_0 \cdot m_1\nonumber\\
&\ge \Big(\frac{\chi_1^2 u^*+\chi_2^2 v^*}{4}\Big)^p\cdot m_0 \cdot  \Big( \frac{m(\mu,\nu,\lambda)}{\delta_0 \cdot |\Omega|\cdot   w^*\cdot  \eta_0 }\Big)^p.
\end{align*}
This implies that
\begin{align*}
\frac{\big(a_{\min}-2\mu h(\chi_1,\chi_2)\big)^{1/p}\delta_0 |\Omega| \min\{\nu,\lambda\}}{m_0^{1/p}}>\frac{\chi_1^2u^*+\chi_2^2 v^*}{4}\frac{m^*}{w^*}\frac{1}
{\eta_0}.
\end{align*}
This together with Theorem \ref{main-thm1} implies that
$$
\Big(\frac{\chi_1^2 u^*}{4}+\frac{\chi_2^2 v^*}{4}\Big)  \frac{m^*(\mu,\nu,\lambda)}{w^* \inf_{t\ge\tau,x\in\Omega} w(t,x) }  <\eta_0\quad \forall\,\, \tau\gg 1.
$$
Therefore,  $A_\tau$ is positive definite
for $\tau\gg 1$. Moreover, there is $\epsilon>0$ such that  \eqref{E-eq1} holds.
The lemma is thus proved.
\end{proof}

We now prove Theorem \ref{main-thm2}.

\begin{proof}[Proof of Theorem \ref{main-thm2}]
(1) First, we prove \eqref{stability-eq0-0}.
By Lemma \ref{stability-lm3},  for $\tau\gg 1$, we have
$$
E(t)-E(\tau)\le -\epsilon\int_\tau ^t \int_\Omega \Big[(u(s,x)-u^*)^2+(v(s,x)-v^*){^2}\Big]dx ds\quad \forall\, t\ge \tau.
$$
This implies that
$$
\epsilon\int_\tau ^t \int_\Omega \Big[(u(s,x)-u^*)^2+(v(s,x)-v^*){^2}\Big]dx ds
\le E(\tau) \quad \forall\, t\ge \tau.
$$
By Lemmas \ref{stability-lm0-1} and \ref{stability-lm0-2},  we have
\begin{equation*}
\int_\Omega \Big[(u(t,x)-u^*)^2+(v(t,x)-v^*)^2\Big]dx \to 0
\end{equation*}
as $t\to\infty$, that is,  \eqref{stability-eq0-0} holds.

Next, we prove \eqref{stability-eq0-1}. Assume that \eqref{stability-eq0-1} does not hold.
Then there are $\epsilon_0>0$ and $t_n\to\infty$ such that
$$
\|u(t_n,\cdot;u_0,v_0)-u^*\|_\infty+\|v(t_n,x;u_0,v_0)-v^*\|_\infty\ge \epsilon_0\quad \forall\, n=1,2,\cdots.
$$
By Proposition \ref{upper-bound-prop}, we may assume that
there are $u^{**}(x),v^{**}(x)$ such that
$$
\lim_{n\to\infty} (\|u(t_n,\cdot;u_0,v_0)-u^{**}(\cdot)\|_\infty+\|v(t_n,\cdot;u_0,v_0)-v^{**}\|_\infty)=0.
$$
By \eqref{stability-eq0-0}, we must have
$$
u^{**}(x)\equiv u^*\quad {\rm and}\quad v^{**}(x)\equiv v^*,
$$
which is a contradiction. Therefore, \eqref{stability-eq0-1} holds and Theorem \ref{main-thm2} is proved.

\smallskip

(2)  Assume that $\chi_1=\chi_2$. Let $\chi=\chi_1$. By \eqref{cond-on-chi-eq2-1}, \eqref{stability-eq5}, and Theorem \ref{main-thm1}(2),   there is $\epsilon>0$ such that
$$
E^{'}(t)\le -\epsilon \Big(\int_\Omega(u-u^*)^2+\int_\Omega (v-v^*)^2\Big)\quad \forall\,\, t\gg 1.
$$
The rest of the proof follows from the same arguments as in (1).
\end{proof}

\end{document}